\documentclass[12pt, 14paper,reqno]{amsart}
\usepackage{amsmath,amsfonts,amssymb}
\usepackage{color}
\makeatletter
\theoremstyle{plain}
\newtheorem{theorem}{Theorem}
\newtheorem{lemma}{Lemma}

\theoremstyle{proof}
\theoremstyle{definition}

\theoremstyle{remark}
\newtheorem{remark}{Remark}
\theoremstyle{lamma}

\numberwithin{equation}{section}
\numberwithin{lemma}{section}
\numberwithin{theorem}{section}
\numberwithin{remark}{section}
\numberwithin{prop}{section}
\numberwithin{corollary}{section}

\usepackage{amsmath}
\usepackage{amsfonts}   
\usepackage{amssymb}
\usepackage{amssymb, amsmath, amsthm}
\usepackage[breaklinks]{hyperref}

\theoremstyle{thmrm}

\numberwithin{conjecture}{section}

\begin{document}
\title{Complete Congruences of Jacobi sums of order $2l^{2}$ with prime $l$}
\author{Md Helal Ahmed and Jagmohan Tanti}
\address{Md Helal Ahmed @ Department of Mathematics, Central University of Jharkhand, Ranchi-835205, India}
\email{ahmed.helal@cuj.ac.in}
\address{Jagmohan Tanti @ Department of Mathematics, Central University of Jharkhand, Ranchi-835205, India}
\email{jagmohan.t@gmail.com}
\keywords{Character; Jacobi sums; Embedding; Congruence's; Cyclotomic field}
\subjclass[2010] {Primary: 11T24, Secondary: 11T22}
\maketitle
\begin{abstract}
The congruences for Jacobi sums of some lower orders have been treated by many authors in the literature. In this paper we establish the 
congruences for Jacobi sums of order $2l^2$ with odd prime $l$ in terms of coefficients of Jacobi sums of order $l$. These congruences are useful to 
obtain algebraic and arithmetic characterizations for Jacobi sums of order $2l^{2}$.
\end{abstract}

\section{Introduction}
Let $e\geq2$ be an integer, $p$ a rational prime,  $q=p^{r}, r \in\mathbb{Z}^{+}$ and $q\equiv 1 \pmod{e}$. Let $\mathbb{F}_{q}$ 
be a finite field of $q$ elements. We can write $q=p^{r}=ek+1$ for some $k\in\mathbb{Z}^{+}$. Let $\gamma$ be a generator of 
the cyclic group $\mathbb{F}^{*}_{q}$ and $\zeta_e=exp(2\pi i/e)$. 
Define a multiplicative character $\chi_e : \ \mathbb{F}^{*}_{q} \longrightarrow \mathbb{Q}(\zeta_e)$ by $\chi_e(\gamma)=\zeta_e$ and extend it on 
$\mathbb{F}_q$ 
by putting $\chi_e(0)=0$. 
For  integers $\displaystyle 0\leq i,j\leq e-1$, the Jacobi sum $J_{e}(i,j)$ is define by

$$J_{e}(i,j)= \sum_{v\in \mathbb{F}_{q}} \chi_e^{i}(v) \chi_e^{j}(v+1).$$

However in the literature a variation of Jacobi sums is also considered and is defined by 
$$J_e(\chi_e^i,\chi_e^j)=\sum_{v\in\mathbb{F}_q}\chi_e^i(v)\chi_e^j(1-v),$$ but are related by $J_e(i,j)=\chi_e^i(-1)J_e(\chi_e^i,\chi_e^j)$.

The study of congruences of Jacobi sums of some small orders is available in the literature. For $l$ an odd prime
Dickson \cite{Dickson} obtained the congruences $J_{l}(1,n)\equiv -1 \pmod{(1-\zeta_{l})^{2}}$ for $1\leq n \leq l-1$. Parnami, Agrawal and Rajwade 
\cite{Parnami} also calculated this separately. Iwasawa \cite{Iwasawa} in $1975$, and in $1981$ Parnami, Agrawal and Rajwade \cite{Parnami 2} 
showed that the above congruences also hold $\pmod{(1-\zeta_{l})^{3}}$. Further in $1995$, Acharya and Katre \cite{Katre} extended the work on finding
the congruences for Jacobi sums and  showed that 
\begin{center}
$J_{2l}(1,n)\equiv -\zeta_l^{m(n+1)}($mod$\ (1-\zeta_{l})^{2})$,
\end{center} 
where $n$ is an odd integer such that $1\leq n \leq 2l-3$ and $m=$ind$_{\gamma}2$.
Also in $1983$, Katre and Rajwade \cite{Rajwade} obtained the congruence of Jacobi sum of order $9$, i.e.,
\begin{center}
$J_{9}(1,1)\equiv -1-($ind$\ 3)(1-\omega)($mod$\ (1-\zeta_{9})^{4})$,
\end{center} 
where $\omega = \zeta_{9}^{3}$.
In $1986$, Ihara \cite{Ihara} showed that if $k>3$ is an odd prime power, then
 \begin{center}
 $J_{k}(i,j)\equiv -1 \pmod{(1-\zeta_{k})^{3}}$.
 \end{center}
 Evans (\cite{Evans}, $1998$) used simple methods to generalize this result for all $k>2$.
Congruences for the Jacobi sums of order $l^{2}$ ($l>3$ prime) were obtained by Shirolkar and Katre \cite{lsquare}. They showed that\\ 
$J_{l^{2}}(1,n)\equiv
 \begin{cases}
  -1 + \sum_{i=3}^{l}c_{i,n} (\zeta_{l^2} -1)^{i} ($mod$\ (1-\zeta_{l^2})^{l+1}) \ \ \ \ \ if\ $gcd$ (l,n)=1, \\
  -1 \ ($mod$\ (1-\zeta_{l^2})^{l+1}) \ \ \ \ \ \ \  \ \ \ \ \ \ \ \ \ \ \ \ \ \ \ \ \ \ \ \ \ \ if\ $gcd$ (l,n)=l.  
 \end{cases}$ \\ \\ 
In this paper, we determine the congruences $\pmod{(1-\zeta_{2l^2})^{l+1}}$ for Jacobi sums of order $2l^{2}$ .  We split the problem into two cases:\\ 
\textbf{Case 1}. $n$ is odd. This case splits into four subcases:\\
 \textbf {Subcase i}. $n=l^{2}$.\\
 \textbf {Subcase ii}. $n=dl, \ $where$\ 1\leq d\leq 2l-1, d$ is an odd and $d\neq l$.\\
 \textbf {Subcase iii}. $1\leq n <2l^2-1$ with gcd$(n,2l^{2})=1$.\\ 
 \textbf {Subcase iv}. $n=2l^2-1$.\\ \\
 
 \noindent \textbf{Case 2}. $n$ is even. In this case the Jacobi sums $J_{2l^{2}}(1,n)$ can be calculated using the relation 
 $J_{2l^{2}}(1,n)=\chi_{2l^2}(-1)J_{2l^{2}}(1,2l^{2}-n-1)$ (which has been shown in the next section).\\ \\  
 Also we calculate here the congruences of Jacobi sums of order $9$ which was not covered in \cite{lsquare} and revised the result of congruences 
 of Jacobi sums of order $l^2$ for $l\geq3$ a prime.
\section{Preliminaries}
Let $\zeta=\zeta_{2l^2}$ and $\chi=\chi_{2l^2}$ then $\chi^2=\chi_{l^2}$ is a character of order $l^2$ and $\zeta_{l^2}=\zeta^2$ is a primitive $l^2$th 
root of unity. The Jacobi sums $J_{l^{2}}(i,j)$ and $J_{2l^{2}}(i,j)$ of order $l^{2}$ and $2l^{2}$ respectively are defined as in the previous section. 
We also have $\zeta=-\zeta_{l^2}^{(l^{2}+1)/2}$.
\subsection{Properties of Jacobi sums} 
In this subsection we discus some properties of Jacobi sums from \cite{Jacobi, lsquare}.
\\ \\
\textbf{Proposition 1}.
If $m+n+s\equiv 0 \ ($mod$\ e)$ then 
\begin{equation*}
J_{e}(m,n)=J_{e}(s,n)=\chi_e^{s}(-1)J_{e}(s,m)=\chi_e^{s}(-1)J_{e}(n,m)=\chi_e^{m}(-1)J_{e}(m,s)=\chi_e^{m}(-1)J_{e}(n,s).
\end{equation*}  	
In particular,
\begin{equation*}
J_{e}(1,m)=\chi_e(-1)J_{e}(1,s)=\chi_e(-1)J_{e}(1,e-m-1).
\end{equation*} 
\textbf{Proposition 2}. 
\ \ \ \ $J_{e}(0,j)= \begin{cases}
-1 \ \ \ \ \ if  \ j\not\equiv 0 \ ($mod$\ e) ,\\
q-2 \ \ if  \ j\equiv 0 \ ($mod$\ e).
\end{cases}$

\ \ \ \ $J_{e}(i,0)=-\chi_e^{i}(-1) \ if \  i\not\equiv 0 \ ($mod$\ e).$
\\ \\ 
\textbf{Proposition 3}.  \ \ \ \ Let $m+n\equiv 0 \ ($mod$\ e)$ \ but not both $m$ and $n$ zero $\pmod{e}$. Then $J_{e}(m,n)=-1.$
\\ \\ \textbf{Proposition 4}.  \ \  For $(k,e)=1$ and $\sigma_k$ a $\mathbb{Q}$ automorphism of $\mathbb{Q}(\zeta_e)$ with $\sigma_k(\zeta_e)=\zeta_e^k$, 
we have $\sigma_{k}J_{e}(m,n)=J_{e}(mk,nk)$. In particular, if $(m,e)=1,\ m^{-1}$ \ denotes the inverse of $m\pmod{e}$ then 
$\sigma_{m^{-1}}J_{e}(m,n)=J_{e}(1,nm^{-1})$.
\\ \\ \textbf{Proposition 5}.  \ \ \ \ 
$J_{2e}(2m,2s)=J_{e}(m,n)$. \\ \\
\textbf{Proposition 6}. \ \ \ Let $m$, $n$, $s$ be integers such that $m+n \not\equiv 0 \ ($mod$\ 2l^{2})$ and $m+s \not\equiv 0 \ ($mod$\ 2l^{2})$. Then 
\begin{equation*}
J_{2l^{2}}(m,n) J_{2l^{2}}(m+n ,s) = \chi^{m}(-1)J_{2l^{2}}(m,s)J_{2l^{2}}(n,s+m).
\end{equation*}
\\  \textbf{Proposition 7}. \ \ \ \ $J_{e}(1,n) \overline{J_{e}(1,n)}= \begin{cases}
q \ \ \ \ \ if \ n\not\equiv 0, -1 \pmod e,\\
1 \ \ \ \ \ if \ n\equiv 0, -1 \pmod e.
\end{cases}$
\begin{proof}
The proofs of 1 - 5 and 7 follows directly using the definition of Jacobi sums (see \cite{Jacobi,lsquare,Wamelen}). The proof of 6 is analogous to the proofs in the 
$2l$ case (see \cite{Katre}). 
\end{proof}
\begin{remark} The Jacobi sums of order $2l^{2}$ can be determined from the Jacobi sums of order $l^{2}$. The Jacobi sums of order $2l^2$
can also be obtained from $J_{2l^{2}}(1,n)$, $1\leq n\leq 2l^{2}-3$ for $n$ odd (or equivalently, $2\leq n\leq 2l^{2}-2$ for $n$ even). Further the Jacobi 
sums of order $l^{2}$ can be evaluated if one knows the Jacobi sums $J_{l^{2}}(1,i)$, $1\leq i\leq \frac{l^{2}-3}{2}$.
\end{remark}

\section{Congruences for Jacobi sums of order $l^{2}$}
The evaluation of congruences for the Jacobi sums of order $l^{2}$, $l>3$ has been done by Shirolkar and Katre \cite{lsquare}. They exclude the case for 
$l=3$. In this section, we settle the case $l=3$ and revise the congruence relations for the Jacobi sums of order $l^{2}$ which shall be used 
to evaluate the congruences relation for the Jacobi sums of order $2l^{2}$ in the next sections.\\,\\ 
For $l=3$, sub-case (2) in Lemma 5.3 \cite{lsquare} reduces to 
$(lx,ly)_{l^{2}}=(ly,lx)_{l^{2}}$, $x\neq y$, $x,y\neq 0$.\\ 
So the number of times these cyclotomic numbers will be counted in all its two forms is
\begin{equation*}
\Bigl\lfloor\dfrac{lx+nly}{l}\Bigr\rfloor+ \Bigl\lfloor\dfrac{ly+nlx}{l}\Bigr\rfloor \equiv 0 \pmod l.
\end{equation*}
So the contribution of $S(n)$ for $l=3$ is same as contribution of $S(n)$ for $l>3$. Hence Theorem 5.4 \cite{lsquare} is precisely revised as the following 
theorem.
\begin{theorem} \label{T1}
Let $l\geq 3$ be a prime and $p^{r}=q\equiv 1 \pmod {l^{2}}$. If $1\leq n \leq l^{2}-1$, then a congruence for $J_{l^{2}}(1,n)$ for a 
finite field $\mathbb{F}_{q}$ is given by \\ $J_{l^{2}}(1,n)\equiv
\begin{cases}
-1 + \sum_{i=3}^{l}c_{i,n} (\zeta_{l^{2}} -1)^{i} \pmod {(1-\zeta_{l^{2}})^{l+1}} \ \ \ \ \ if\ \emph{gcd} (l,n)=1, \\
 -1 \pmod{(1-\zeta_{l^{2}})^{l+1}} \ \ \ \ \ \ \  \ \ \ \ \ \ \ \ \ \ \ \ \ \ \ \ \ \ \ \ \ \ \ if\ \emph{gcd} (l,n)=l,  
\end{cases}$  \\ 
where for $3\leq i \leq l-1$, $c_{i,n}$ are described by equation (5.3) and $c_{l,n}=S(n)$ is given by Lemma 5.3 in \cite{lsquare}.
\end{theorem}
\section{Some Lemmas}
\begin{lemma} \label{L1}
Let $p\geq 3$ be a prime and $q=p^{r}\equiv 1 \pmod{2l^{2}}$. If $\chi$ is a nontrivial character of order $2l^{2}$ on the finite field $\mathbb{F}_{q}$, 
then
\begin{equation*}
J_{2l^{2}}(a,a)=\chi^{-a}(4)J_{2l^{2}}(a,l^{2}).
\end{equation*}
\end{lemma}
\begin{proof}
As for $\alpha\in\mathbb{F}_{q}$ the number of $\beta\in\mathbb{F}_{q}$ satisfying the equation $\beta(1+\beta)=\alpha$ is same as 
$1+\chi^{l^{2}}(1+4\alpha)$, we have 
\begin{eqnarray*}
J_{2l^{2}}(a,a)&=&\sum_{\beta\in\mathbb{F}_{q}}\chi^{a}(\beta(1+\beta))= \sum_{\alpha\in\mathbb{F}_{q}}\chi^{a}(\alpha)\{1+\chi^{l^{2}}(1+4\alpha)\} \\
&=& \chi^{-a}(4)\sum_{\alpha\in\mathbb{F}_{q}}\chi^{a}(4\alpha)\chi^{l^{2}}(1+4\alpha) 
= \chi^{-a}(4) J_{2l^{2}}(a,l^{2}).
\end{eqnarray*}
\end{proof}
\begin{lemma} \label{L2}
Let $l\geq 3$ be a prime and $q=p^{r}\equiv 1 \pmod{2l^{2}}$, then 
\begin{equation*}
J_{2l^{2}}(1,l^{2})\equiv \zeta_{l^{2}}^{-w}(-1+\sum_{i=3}^{l} c_{i,(l^{2}-1)/2}(\zeta_{l^{2}}-1)^{i}) \pmod {(1-\zeta_{l^{2}})^{l+1}},
\end{equation*}
where $w=\rm{ind}_\gamma2$ with $\gamma$ a generator of $\mathbb{F}_q^*$.
\end{lemma}
\begin{proof}
From Proposition 6 for $m=1$, $n=1$ and $s=l^{2}-1$, we have
\begin{equation} \label{5.1}
J_{2l^{2}}(1,1) J_{2l^{2}}(2 ,l^{2}-1) = \chi(-1)J_{2l^{2}}(1,l^{2}-1)J_{2l^{2}}(1,l^{2}).
\end{equation}
By Proposition 1, we obtain $\chi (-1) J_{2l^{2}}(1, l^{2}-1)= J_{2l^{2}}(1,l^{2})$.\\ 
Now equation (\ref{5.1}) becomes 
\begin{equation} \label{5.2}
J_{2l^{2}}(1,1) J_{2l^{2}}(2, l^{2}-1)= J_{2l^{2}}(1,l^{2})J_{2l^{2}}(1,l^{2}).
\end{equation} 
Again by Proposition 5 and Theorem \ref{T1}, we have
\begin{equation} \label{5.3}
J_{2l^{2}}(2, l^{2}-1)=J_{l^{2}}(1, (l^{2}-1)/2)\equiv -1 + \sum_{i=3}^{l}c_{i,(l^{2}-1)/2} (\zeta_{l^{2}} -1)^{i} \pmod{(1-\zeta_{l^{2}})^{l+1}}. 
\end{equation}
For $w=ind_{\gamma}2$, from \textbf{Lemma \ref{L1}}, we have
\begin{equation} \label{5.4}
J_{2l^{2}}(1,1)=\chi^{-1}(4)J_{2l^{2}}(1,l^{2})=\zeta_{l^{2}}^{-w}J_{2l^{2}}(1,l^{2}).
\end{equation} 
Employing (\ref{5.4}) and (\ref{5.3}) in (\ref{5.2}), we get
\begin{equation*}
J_{2l^{2}}(1,l^{2})\equiv \zeta_{l^{2}}^{-w}(-1+\sum_{i=3}^{l} c_{i,(l^{2}-1)/2}(\zeta_{l^{2}}-1)^{i})\pmod{(1-\zeta_{l^{2}})^{l+1}}.
\end{equation*}
\end{proof}
\begin{lemma} \label{L3}
Let $n$ be an odd integer such that $1\leq n < 2l^{2}-1$ and $\emph{gcd}(n,2l^{2})=1$, then
\begin{align*}
J_{2l^{2}}(1,n)\equiv & \zeta_{l^{2}}^{-w(n+1)}(-1+\sum_{i=3}^{l} c_{i,(l^{2}-1)/2}(\zeta_{l^{2}}-1)^{i})(-1+\sum_{i=3}^{l} c_{i,(l^{2}-1)/2}(\zeta_{l^{2}}^{n}-1)^{i})\\ &(-1+\sum_{i=3}^{l} c_{i,(-1-n)}
(\zeta_{l^{2}}^{(1-l^{2})/2}-1)^{i})\pmod{(1-\zeta_{l^{2}})^{l+1}},
\end{align*}
where $w=\rm{ind}_\gamma2$ with $\gamma$ a generator of $\mathbb{F}_q^*$.
\end{lemma}
\begin{proof}
By Proposition 6 for $m=1$ and $s=l^{2}-1$, we have
\begin{equation*}
J_{2l^{2}}(1,n) J_{2l^{2}}(1+n ,l^{2}-1) = \chi(-1)J_{2l^{2}}(1,l^{2}-1)J_{2l^{2}}(n,l^{2}).
\end{equation*}
Applying Proposition 1 to get
\begin{equation*}
J_{2l^{2}}(1,n) J_{2l^{2}}(1+n ,l^{2}-1) = J_{2l^{2}}(1,l^{2})J_{2l^{2}}(n,l^{2}).
\end{equation*}
Applying $\tau_{n}:\zeta_{2l^{2}}\rightarrow\zeta_{2l^{2}}^{n}$ (a $\mathbb{Q}$ automorphism of $\mathbb{Q}(\zeta_{2l^2})$) second term in RHS, we get
\begin{equation*}
J_{2l^{2}}(1,n) J_{2l^{2}}(1+n ,l^{2}-1) =J_{2l^{2}}(1,l^{2})\tau_{n}J_{2l^{2}}(1,l^{2}).
\end{equation*}
Now as from \cite{Jacobi,lsquare,Wamelen}, we have 
$J_{2l^{2}}(1+n,l^{2}-1) \overline{J_{2l^{2}}(1+n,l^{2}-1)}=q$ and so, we have
\begin{equation*}
J_{2l^{2}}(1,n)J_{2l^{2}}(1+n,l^{2}-1) \overline{J_{2l^{2}}(1+n,l^{2}-1)}= J_{2l^{2}}(1,l^{2})\tau_{n}J_{2l^{2}}(1,l^{2})
\overline{J_{2l^{2}}(1+n,l^{2}-1)}.
\end{equation*}
This implies
\begin{equation}  \label{6.5}
J_{2l^{2}}(1,n)\ q= J_{2l^{2}}(1,l^{2})\tau_{n}J_{2l^{2}}(1,l^{2})\overline{J_{2l^{2}}(1+n,l^{2}-1)}.
\end{equation}
From Proposition 5 and applying $\sigma_{n}:\zeta_{l^{2}}\rightarrow\zeta_{l^{2}}^{n}$ (a $\mathbb{Q}$ automorphism of $\mathbb{Q}(\zeta_{l^2})$), we get,
\begin{align*}
\overline{J_{2l^{2}}(1+n,l^{2}-1)}=J_{2l^{2}}(-1-n,1-l^{2})=J_{l^{2}}((-1-n)/2,(1-l^{2})/2)=\sigma_{(1-l^{2})/2}(J_{l^{2}}(1,-1-n).
\end{align*}
Now from Theorem \ref{T1}, we get
\begin{equation} \label{6.4}
\overline{J_{2l^{2}}(1+n,l^{2}-1)}\equiv-1+\sum_{i=3}^{l} c_{i,(-1-n)}(\zeta_{l^{2}}^{(1-l^{2})/2}-1)^{i}\pmod {(1-\zeta_{l^{2}})^{l+1}}.
\end{equation}
Employing (\ref{6.4}) and \textbf{Lemma \ref{L2}} in (\ref{6.5}), we get
\begin{align*}
J_{2l^{2}}(1,n)\equiv & \zeta_{l^{2}}^{-w(n+1)}(-1+\sum_{i=3}^{l} c_{i,(l^{2}-1)/2}(\zeta_{l^{2}}-1)^{i})(-1+\sum_{i=3}^{l} c_{i,(l^{2}-1)/2}(\zeta_{l^{2}}^{n}-1)^{i})\\ &(-1+\sum_{i=3}^{l} c_{i,(-1-n)}
(\zeta_{l^{2}}^{(1-l^{2})/2}-1)^{i})\pmod{(1-\zeta_{l^{2}})^{l+1}}.
\end{align*}
\end{proof}
\begin{remark} \label{remark1}
For $n = 2l^{2}-1$ by Proposition 3, we have $J_{2l^{2}}(1,n)\equiv -1$. 
\end{remark}
\begin{lemma} \label{L5}
Let $d\neq l$, $1\leq d \leq 2l-1$ be an odd positive integer, then
\begin{align*} J_{2l^{2}}(1,dl)\equiv-\zeta_{l^{2}}^{-w(dl+1)}(-1+\sum_{i=3}^{l} c_{i,(l^{2}-1)/2}(\zeta_{l^{2}}-1)^{i})(-1+\sum_{i=3}^{l} c_{i,dl-1}(\zeta_{l^{2}}^{(-1-dl)/2}-1)^{i})\pmod{(1-\zeta_{l^{2}})^{l+1}},
\end{align*}
where $w=\rm{ind}_\gamma2$ with $\gamma$ a generator of $\mathbb{F}_q^*$.
\end{lemma}
\begin{proof}
From Proposition 6, for $m=1$, $n=dl$ and $s=l^{2}-1$, we have
\begin{equation} \label{5.7}
J_{2l^{2}}(1,dl) J_{2l^{2}}(1+dl ,l^{2}-1) = \chi(-1)J_{2l^{2}}(1,l^{2}-1)J_{2l^{2}}(dl,l^{2}).
\end{equation}
Applying Proposition 1, we get
\begin{equation} \label{5.8}
J_{2l^{2}}(1,dl) J_{2l^{2}}(1+dl ,l^{2}-1) = J_{2l^{2}}(1,l^{2})J_{2l^{2}}(dl,l^{2}).
\end{equation} 
Now again from \cite{Jacobi,lsquare,Wamelen}, we have
$J_{2l^{2}}(1+dl,l^{2}-1) \overline{J_{2l^{2}}(1+dl,l^{2}-1)}=q$ and so
we have
\begin{equation*}
J_{2l^{2}}(1,dl) J_{2l^{2}}(1+dl ,l^{2}-1)\overline{J_{2l^{2}}(1+dl,l^{2}-1)} = J_{2l^{2}}(1,l^{2})J_{2l^{2}}(dl,l^{2})\overline{J_{2l^{2}}(1+dl,l^{2}-1)}.
\end{equation*}
This implies
\begin{equation}  \label{6.3}
J_{2l^{2}}(1,dl)\ q = J_{2l^{2}}(1,l^{2})J_{2l^{2}}(dl,l^{2})\overline{J_{2l^{2}}(1+dl,l^{2}-1)}.
\end{equation}
By Proposition 5, we get 
\begin{align*}
\overline{J_{2l^{2}}(1+dl,l^{2}-1)}=J_{2l^{2}}(-1-dl,1-l^{2})=J_{l^{2}}((-1-dl)/2,(1-l^{2})/2)=\sigma_{(-1-dl)/2}(J_{l^{2}}(1,dl-1).
\end{align*}
Now from Theorem \ref{T1}, we get
\begin{equation} \label{6.6}
\overline{J_{2l^{2}}(1+dl,l^{2}-1)}\equiv-1+\sum_{i=3}^{l} c_{i,dl-1}(\zeta_{l^{2}}^{(-1-dl)/2}-1)^{i}\pmod{(1-\zeta_{l^{2}})^{l+1}}.
\end{equation}
As $\chi$ is of order $2l^{2}$, $\chi^{l}$ is of order $2l$, so by Lemma $3$ \cite{Katre}, we obtain 
\begin{align*}
J_{2l^{2}}(dl,l^{2})&=J_{2l}(d,l)=(\chi^{l})^{d}(4)J_{2l}(d,d)=\chi^{2ld}(2)\eta_{d}(J_{2l}(1,1)).
\end{align*}
For $w=ind_{\gamma}2$, from Proposition 3 \cite{Katre}, we get
\begin{align*}
J_{2l}(1,1)\equiv-\zeta_{l}^{-2w}\pmod {(1-\zeta_{l})^{2}}.
\end{align*}
Applying $\eta_{d}:\zeta_{l}\rightarrow\zeta_{l}^{d}$ (a $\mathbb{Q}$ automorphism of $\mathbb{Q}{\zeta_l}$), we get
\begin{align*}
\eta_{d}(J_{2l}(1,1))\equiv-\zeta_{l}^{-2wd}\pmod {(1-\zeta_{l}^{d})^{2}}\equiv-\zeta_{l}^{-2wd}\pmod{(1-\zeta_{l})^{2}}.
\end{align*}
Thus 
\begin{align*}
J_{2l^{2}}(dl,l^{2})&\equiv\zeta_{l}^{dw}(-\zeta_{l}^{-2wd})\pmod{(1-\zeta_{l})^{2}}\equiv-\zeta_{l}^{-wd}\pmod {(1-\zeta_{l})^{2}}\\ & \equiv-\zeta_{l^{2}}^{-wdl}\pmod{(1-\zeta_{l^{2}}^{l})^{2}}\equiv-\zeta_{l^{2}}^{-wdl}\pmod{(1-\zeta_{l^{2}})^{2}}.
\end{align*}
This implies 
\begin{equation}\label{5.9}
J_{2l^{2}}(dl,l^{2})(1-\zeta_{l^{2}})^{l-1}\equiv-\zeta_{l^{2}}^{-wdl}(1-\zeta_{l^{2}})^{l-1} \pmod {(1-\zeta_{l^{2}})^{l+1}}.
\end{equation}
Employing  (\ref{6.6}), (\ref{5.9}) and Lemma \ref{L2} in (\ref{6.3}), we get
\begin{align*} J_{2l^{2}}(1,dl)\equiv-\zeta_{l^{2}}^{-w(dl+1)}(-1+\sum_{i=3}^{l} c_{i,(l^{2}-1)/2}(\zeta_{l^{2}}-1)^{i})(-1+\sum_{i=3}^{l} c_{i,dl-1}(\zeta_{l^{2}}^{(-1-dl)/2}-1)^{i}) \pmod {(1-\zeta_{l^{2}})^{l+1}}.
\end{align*}
\end{proof}
\section{Main Theorem}
In this section we establish the congruences of Jacobi sums of order $2l^2$ in terms of coefficients of Jacobi sums of order $l$.
\begin{theorem} 
Let $l\geq 3$ be a prime and $q=p^{r}\equiv 1 \pmod {2l^{2}}$. If $1\leq n \leq 2l^{2}-1$ and $1\leq d \leq 2l-1$ are odd integer, then a congruence for $J_{2l^{2}}(1,n)$ over $\mathbb{F}_{q}$ is given by
\begin{equation*}
J_{2l^{2}}(1,n)\equiv \begin{cases}
\zeta_{l^{2}}^{-w}(-1+\sum_{i=3}^{l} c_{i,(l^{2}-1)/2}(\zeta_{l^{2}}-1)^{i}) \pmod {(1-\zeta_{l^{2}})^{l+1}},\ \mbox{if $n=l^{2}$}, \\ \\ 
-\zeta_{l^{2}}^{-w(dl+1)}(-1+\sum_{i=3}^{l} c_{i,(l^{2}-1)/2}(\zeta_{l^{2}}-1)^{i})(-1+\sum_{i=3}^{l} c_{i,dl-1}(\zeta_{l^{2}}^{(-1-dl)/2}-1)^{i})\\
\pmod {(1-\zeta_{l^{2}})^{l+1}}  \ \mbox{if $d\neq l$ odd integer and $n=dl$},\\ \\
  \zeta_{l^{2}}^{-w(n+1)}(-1+\sum_{i=3}^{l} c_{i,(l^{2}-1)/2}(\zeta_{l^{2}}-1)^{i})(-1+\sum_{i=3}^{l} c_{i,(l^{2}-1)/2}(\zeta_{l^{2}}^{n}-1)^{i})\\ (-1+\sum_{i=3}^{l} c_{i,(-1-n)}
  (\zeta_{l^{2}}^{(1-l^{2})/2}-1)^{i})\pmod {(1-\zeta_{l^{2}})^{l+1}},\, \mbox{if \, $gcd(n,2l^{2})=1$, $1\leq n < 2l^{2}-1$}, \\ \\
  -1 \pmod {(1-\zeta_{l^{2}})^{l+1}},\, \mbox{if \, $n = 2l^{2}-1$},  
\end{cases}
\end{equation*}
where $c_{i,n}$ are as described in the Theorem \ref{T1} and $w=\rm{ind}_\gamma2$ with $\gamma$ a generator of $\mathbb{F}_q^*$.
\vskip2mm
If $n$ is even, $2\leq n\leq 2l^{2}-2$ the congruences for Jacobi sums $J_{2l^{2}}(1,n)$ can be calculated using the relation 
$J_{2l^{2}}(1,n)=\chi(-1)J_{2l^{2}}(1,2l^{2}-n-1)$. Also if $d$ in the theorem is even then $J_{2l^2}(1,dl)=\chi(-1)J_{2l^2}(1, 2l^2-dl-1)$ and $2l^2-dl-1$
is odd. Thus the congruences for $J_{2l^2}(1,n)$ gets completely determined and hence that of all Jacobi sums of order $2l^2$.
\end{theorem}
\begin{proof}
The proof of the theorem is immediate from the above-mentioned Lemma's and remark \ref{remark1}.
\end{proof}

\begin{remark}
The prime ideal decompositions and absolute values of Jacobi sums are already there in the literature, adding these congruences gives idea to determine 
Jacobi sums of order $2l^2$ with less complexity. P. Van Wamelen \cite{Wamelen} has developed an idea to establish certain algebraic conditions  
using which one can determine the Jacobi sums uniquely. Our congruence conditions are in terms of coefficients of Jacobi sums of order $l$ and are modulo 
$(1-\zeta_{l^2})^{l+1}$. These congruences are appropriate and have deterministic capacity for Jacobi sums of order $2l^2$.
\end{remark}
\begin{remark}
Establishing congruences for Jacobi sums is the significant advancement in a more straightforward solution of the Jacobi sums problem and thus for the cyclotomic problem. L. E. Dickson laid the foundation stone of cyclotomic numbers and he demonstrated how the Jacobi sums play a significant role in this theory. In \cite{Helal1}, we have given the formula for cyclotomic numbers of order $2l^2$ in terms of the coefficients of Jacobi sums of order $2l^2$. Therefore, these congruences are useful in calculations of these Jacobi sums. 
\end{remark}
\begin{remark}
	The Jacobi sums and cyclotomic numbers have incredible applications in various field, such as coding theory, cryptosystems, primality testing, difference sets, and so forth \cite{Adleman1,Buhler1,Mihailescu2,Whiteman1}. Thus congruences give us an analogous result to evaluate Jacobi sums, so cyclotomic numbers.
\end{remark}
\begin{remark}
	It remains an open problem to develop a methodology to determine Jacobi sums using Stickelberger's theorem \cite{Adhikari1} along with these determined congruences.
\end{remark}

\section*{Acknowledgments}
\noindent
The authors would like to thank Central University of Jharkhand, Ranchi, Jharkhand, India for the support during preparation of this research article.

\end{document}